\documentclass[11pt,oneside]{amsart}   	% use "amsart" instead of "article" for AMSLaTeX format
\usepackage{geometry}                		% See geometry.pdf to learn the layout options. There are lots.
\usepackage{graphicx}				% Use pdf, png, jpg, or eps§ with pdflatex; use eps in DVI mode
\usepackage{amssymb} 
\usepackage{amsmath}
\usepackage{amsthm}
\usepackage{booktabs}
\usepackage{placeins}

\newtheorem{theorem}{Theorem}
\newtheorem{corollary}[theorem]{Corollary}

\newtheorem{lemma}[theorem]{Lemma}

\newtheorem{definition}{Definition}[section]

\newtheorem{remark}{Remark}

\title{Set Turing Machines}
\author{Garvin Melles}
\begin{document}
\maketitle

\section{Introduction}

In this paper we define a notion of Turing computability for class functions, i.e., functions that operate on arbitrary sets. 
We generalize the notion of a Turing machine to the set Turing machine. 
Set Turing machines operate on a class size tape. We represent sets by placing marks in the cells of the set Turing machine tape. 
Instead of being indexed by $\mathbb{N}$ or $\mathbb{Z}$,
the tapes cells are indexed by finite sequences of ordinals.
For a marking of the cells to represent a set, the markings have the structure of a tree which mirrors the  transitive closure of the set.

Our conception depends on both the Axiom of Choice and the Axiom of Foundation.
Representations of sets as marks on the set Turing machine tape exist by the Axiom of Choice. The representations are well founded by the
Axiom of Foundation.

Using the concepts of the set Turing machines and the encoding of sets by marks on the set Turing machine tape, 
we  define the Turing computable class functions denoted $\rm{TUR}$.  
We also define the collection of recursive class functions, $\rm{REC}$, a generalization of the primitive recursive set functions
as defined in \cite{JK}. The class functions in ${\rm{REC}}$ are analogous to the 
recursive functions on $\mathbb N$. We will prove some elementary properties about 
${\rm{REC}}$. In the last section we prove our main theorem that $\rm{TUR}=\rm{REC}$.

\section{Set Turing Machines}

\subsection{Overview}

In this section we overview our definition without full formality for  a generalization of the Turing machine, the set Turing machine.
A set Turing machine  is an analog of the Turing machine
except instead of operating $n$ tuples of natural numbers it operates on $n$ tuples of sets.   Any $n$ tuple of well founded sets can be
represented by an $n$-tuple of well founded trees, each tree mirroring the transitive closure of a set. 
In our definition, we had the choice of whether to work with operands that are sets or to work with operands that
are $n$-tuples of sets. We find it convenient to work with the later. 

A standard Turing machine consists of a one way or two way infinite tape
consisting of infinitely many cells together with a finite set of instructions for controlling the stepping of  the machine. 
We can think of a one way tape to have elements of $\mathbb{N}$ for addressing the tape's cells  or a two way tape to have
elements of $\mathbb{Z}$ for addresses of its cells.  As a first step toward developing the set Turing machine, we consider a Turing
machine whose tape consists of cells indexed by elements of $^{<\omega}\omega$. Such a machine, along with a richer
set of machine head movements,  operates naturally on $n$-tuples from
$HC$, the hereditary finite sets.
We allow machine head movements
\begin{enumerate}
\item $z$ = move the head from cell $(n_0,\ldots,n_k)$ to cell $(n_0,\ldots,n_k,0)$, 
\item $u$ = move the head from cell $(n_0,\ldots,n_k)$ to cell
$(n_0,\ldots,n_{k-1})$, 
\item $+$ = move the head from cell $(n_0,\ldots,n_k)$ to cell $(n_0,\ldots,n_k+1)$, 
\item $u+$ = move the head from cell $(n_0,\ldots,,n_{k-1}, n_k)$ to cell $(n_0,\ldots,n_{k-1}+1)$,
\item $j^+$ (jump plus) = move the head from cell $(n_0,\ldots,n_k)$ to cell $(n_0+1,\ldots,n_k)$, and 
\item $j^-$ (jump minus) = move the head from cell $(n_0,\ldots,n_k)$  to cell $(n_0-1,\ldots,n_k)$. 
\end{enumerate}

One of the main technical points of this
paper is to show how such finite machines (along with a  class size address space) can be considered to operate on arbitrary sets. 
For infinite inputs, we have to consider how the machine operates through limit steps and how to set up the tape to encode
the set. For a set Turing machine to operate on arbitrary sets, the cells of the tape are indexed by 
tuples of the form $\nu=(n,\bar{\alpha})$ (also written as $n^\frown\bar{\alpha}$) where 
$n\in\omega$ and $\bar{\alpha}\in \,^{<\omega}ON$, the set of finite sequences of ordinals.  
In each cell a mark from a finite alphabet can be placed on or erased from the cell.
Given a well ordering of the transitive closure of a set $x$, we can create a marking $X$ of the set Turing machine tape that encodes
$x$.
Without loss of generality the marks, will be from the set $\{1, 2, 3, 4,\star, \star\star\}$ with a blank identified with $0$. 
We also allow the placing of a $^{\star}$ superscript on cells marked with $0,1,2,3,4$ as a kind of mark on a mark. These will be used
to guide the set Turing machine. 

We encode a $n$- tuple of sets $x_0,\ldots,x_{n-1}$ by marking some set of the cells with $1$'s to form a well founded tree. 
For $i\in n$, the markings that have cells with addresses of the form $i^{\frown}\bar{\alpha}$ will represent $i$-th component
of the tuple. Markings other than $1$ are used to guide set Turing machines in their operation. Only certain sets of markings are proper encodings of tuples of sets. While there are many ways to properly mark a set Turing machine tape to encode a given set, given such a marking, the
decode process is well defined, that is the $Decode(v)$ function can take a properly marked set Turing machine tape $X$ and
output the unique set $Decode(X)$ which is encoded by $X$. We say a class function $F:V^n\rightarrow V$ is Turing computable
if there is a set Turing machine $M$ which computes it, that is for every $n$-tuple, $x_0,\ldots,x_{n-1}$ and every choice of encoding
say  $X_0,\ldots, X_{n-1}$ of $x_0,\ldots,x_{n-1}$, the output of $M$  on input $X_0,\ldots, X_{n-1}$, denoted $M(X_0,\ldots, X_{n-1})$,
 is defined, well formed and 
$$
Decode(M(X_0,\ldots, X_{n-1})=F(x_0,\ldots,x_{n-1}).
$$

As with ordinary Turing machines one can define the configuration of a set Turing machine. 
The configurations of the set Turing machine consists of an
ordinal $\alpha\in ON$, representing the time step, the machine head position $\nu=n^{\frown} \bar{\alpha}$, 
which will be some finite sequence of ordinals whose first
component is in $\omega$,
a marking $X$ of the tape, which will
have some set of cells of the tape having non zero  marks, and a state $\ell\in\omega$ of the machine, which is just a finite index to a
row of the finite  set Turing machine table.  We denote a generic configuration $C_{\alpha}$ at time step $\alpha$ by a $4$-tuple
$$C_{\alpha}=(\alpha, \nu, X, \ell)$$
where $\nu\in \omega\times^{<\omega}ON$ and $X=(X_0,\ldots,X_{n-1})$ where $X_i$ is the subset of markings in $X$ whose
address begins with $i$.

A set Turing machine operates in a  completely deterministic manner, with changes in a configuration after
a single step defined by the set Turing machine table. The Turing machine head
moves to a neighboring cell, except for the jump instruction which takes it to the corresponding cell a neighboring component. 
For a limit ordinal $\beta$, we define the limit configuration of the machine at step $\beta$. The
limit configuration (if it exists) is completely
determined by the configurations of the machine at steps $\alpha<\beta$. We give more details on what limit configurations are
in the next paragraph.

The cells of the set Turing machine tape are well ordered via the canonical lexicographic
well ordering $\prec$ of $^{<\omega}ON$. We denote generic elements of $\omega\times ^{<\omega}ON$ by $\nu$, $\rho$, $\eta$. 
If $\nu$ is an initial segment of $\rho$, we write $\nu\vartriangleleft\rho$. Then also $\nu\prec \rho$. The state of the
machine at limit step $\beta$ will be the least state which occurs cofinally in the configurations leading up to the
$\beta$-th. Similarly, if $X_{\alpha}$ denotes the marked tape at time $\alpha$, then the markings  of the tape at time
$\beta$ is given by 
$$
X_{\beta}=\lim_{\alpha\rightarrow \beta}X_{\alpha}
$$
defined in the natural way. If this limit does not exist, then the set Turing machine fails to compute an output at time step $\beta$. 
Finally, we must consider the movement of the set Turing machine head. If the sequence of head positions up to time step 
$\beta$ has a limit, then the $\beta$-th head
position is this limit. Otherwise, it is the minimal in the sense of $\prec$ node which occurs cofinally among the head positions
of the set  Turing machine leading up to step $\beta$. 

In addition to these above considerations, we must define how to encode a set by a set of marks on the tape, what it means to
be a well formed tape, what it means for a set Turing machine to produce an output for a given input. 
In general, our computable class functions will be partial, i.e. class functions $F$
 such that
$$
{\rm dom} \,F\subseteq V^n
$$
for some $n\in \mathbb{N}$. Throughout the paper, we use $0$-based indexing for finite sequences.

\subsection{Set Turing Machines in More Formal Detail}

In this subsection we define the set Turing machine with  more formality and detail. 

\begin{definition}
For $\eta\in \,^{<\omega}ON$ let ${\rm len}(\eta)$ denote its length.
For $\eta,\nu\in\,^{<\omega}ON$, if $\eta$ is an initial segment of $\nu$, we write $\eta\vartriangleleft\nu$.
We define the lexicographic order $\prec_{lex}$ on $^{<\omega}ON$ by letting $\eta\prec_{lex}\nu$ if and only if $\eta\vartriangleleft\nu$
or if $i\in\omega$ is the first index
such that $\eta(i)\neq \nu(i)$ then $\eta(i)<\nu(i)$. We also denote $\prec_{lex}$ by $\prec$. 
\end{definition}

\begin{definition}
A set $X\subseteq\, ^{<\omega}ON$ is a basic code if 
\begin{enumerate}
\item If $\nu\in X$ and $\rho\vartriangleleft\nu$, then $\rho\in X$. In particular, $\langle\rangle$, the empty sequence is in $X$. 
\item If $\bar{\alpha}\in\,X$, {\rm then for no} $\alpha_1<\alpha_2$ {\rm is} $\bar{\alpha}^\frown \alpha_1\not\in X$, 
$\bar{\alpha}^\frown \alpha_2\in X$.
\end{enumerate}
\end{definition}

\begin{definition}
If $X$ is a basic code then $(X,\vartriangleleft)$ is a well ordered set. So we define a rank function on $X$
$$
\rho_{(X,\vartriangleleft)}:X\rightarrow ON
$$
in the usual way. 
\end{definition}

\begin{definition}
Let $X$ be a basic code. We define the decode function
$$
decode_X:X\rightarrow V
$$
for elements $x\in X$ by induction on $\rho_{(X,\vartriangleleft)}(x)$.
If $\rho_{(X,\vartriangleleft)}(x)=0$, then $decode_X(x)=\emptyset$.
Otherwise, 
$$
decode_{X}(x)=\{decode_{X}(y)\mid x\vartriangleleft y\ \wedge\ {\rm len}(y)={\rm len}(x)+1\}.
$$
\end{definition}

\begin{definition}
If $X$ is a basic code, then we define
$$
Decode(X)=decode_X(\langle\rangle).
$$
\end{definition}

\begin{definition}
If $X_1$ and $X_2$ are basic codes, then we say $X_1$ and $X_2$ are equivalent, $X_1\sim X_2$
if and only if 
$$
(X_1,\vartriangleleft)\cong (X_2,\vartriangleleft)
$$
\end{definition}

\begin{definition}
If $x$ is a set and $X$ is a basic code such that
$$
(trcl(\{x\}),\in)\cong (X,\vartriangleleft)
$$
then we say that $X$ is a basic code for $x$. 
\end{definition}

\begin{lemma}
If $X$ is a basic code, then there is a unique set $x$ such that $X$ is a basic code for $x$. 
Furthermore, if $X_1$ and $X_2$ are two basic codes, then there is a $x$ which is a basic
code for both $X_1$ and $X_2$ if and only if $X_1\sim X_2$.
\end{lemma}

\begin{lemma}
Let $X\subseteq \,^{<\omega}ON$ be a basic code for $x$ via the isomorphism  $f$, i.e.,
$$
(trcl(\{x\}),\in)\cong_f (X,\vartriangleleft).
$$
Then the rank function $\rho_{(X,\vartriangleleft)}$ on $X$ preserves rank, that is for all $y\in trcl(\{x\})$,
$$
{\rm rank}(y)=\rho_{(X,\vartriangleleft)}(f(y))
$$
where ${\rm rank}$ refers to set theoretic rank.
\end{lemma}

\begin{definition}
If $x=(x_0,\ldots,x_{n-1})$ is a finite sequence of sets, then $X=(X_0,\ldots,X_{n-1})$ is a code for $x$ if 
for each $i$ such that  $i\in n$,
$$
X_i=\{\bar{\alpha}\mid i^{\frown}\bar{\alpha}\in X\}
$$
is a basic code for $x_i$. Furthermore, we define the Decode of $X=(X_0,\ldots,X_{n-1})$ via
$$
Decode(X)=(Decode(X_0),\ldots,Decode(X_{n-1})).
$$
\end{definition}

\noindent We will formally identify a blank cell with a cell marked with $0$. We have

\begin{definition}
The set Turing machine tape is the class $T=\mathbb{N}\times\, ^{<\omega}ON$. A set Turing machine marking is a class function
$X$ with domain $T$,
$$
X:T\rightarrow \{0,1,2,3,4,0^{\star},1^{\star},2^{\star},3^{\star},4^{\star},\star,\star\star\}
$$ 
such that the inverse image of $\{1,2,3,4,0^{\star},1^{\star},2^{\star},3^{\star},4^{\star},\star,\star\star\}$, 
$${X}^{-1}(\{1,2,3,4,0^{\star},1^{\star},2^{\star},3^{\star},4^{\star},\star,\star\star\})$$
is a set. A marking $X$ is well formed if ${X}^{-1}(\{1,2,3,4,0^{\star},1^{\star},2^{\star},3^{\star},4^{\star},\star,\star\star\})$
is a code.
\end{definition}

\noindent If $X$ is a well formed tape marking such that the range of $X$ is $\{0,1\}$, then we identify $X$ with $X^{-1}(1)$, the
code for a tuple of sets. We will also call $X$ a Turing name for the set it encodes, namely, $Decode(X)$. 

\begin{definition}
Let $X$ be a marking. For  $\eta\in T$, we define the induced marking $X_{\eta}$ by letting
$X_{\eta}(\bar{\alpha})=X(\eta^{\frown}\bar{\alpha})$ for finite sequences of ordinals $\bar{\alpha}$,
and let $X_{\eta}(\langle\rangle)=X(\eta)$.
\end{definition}

\begin{definition}
If $X$ is a well formed marking then the number of components, $$\#components(X)$$ is the smallest $n\in\omega$ such that
if $k\in\omega$ and $(k,\bar{\alpha})\in T$ with $X(k,\bar{\alpha})\neq 0$, then $k<n$. For $k<\#components(X)$,
by the $k$-th component $X_k$  we mean $\{\bar{\alpha}\in\, ^{<\omega}ON\mid X(k,\bar{\alpha})\neq 0\}$. 
\end{definition}

\begin{definition}
A root of the Turing machine tape
is a cell whose index is a natural number.
\end{definition}

The notion of rank extends naturally to well formed markings. 

\begin{definition}
If $X$ has components $X_0,\ldots,X_{n-1}$ then we define the
rank of $X$ as the sup of the ranks of each component. For a component $X_i=Y$, its rank
is the rank of the basic code it is identified with.
\end{definition}

\begin{definition}
A set Turing machine move $mv$ is an element of 
$$\{s ,z, u, +, u+, j+, j-\}$$
with the following meanings.
\begin{enumerate}
\item $s$ (for same) denoting no change in cell address, 
\item $z$ (for zero) adjoin zero to the cell address, 
\item $u$ (for up) move up the tree by deleting the last component of the
cell address, 
\item $+$ add one to the last component of the cell address, 
\item $u+$ move up the tree and then move over to the next cell address.
\item $j+$  jump to the next component, 
\item $j-$ negative jump to the previous component respectively. 
\end{enumerate}
\end{definition}

\begin{definition}
Let $\{0,1,2,3,4,0^{\star},1^{\star},2^{\star},3^{\star},4^{\star},\star, \star\star\}$ be the set of marks one can place in a set Turing machine cell.
A set Turing machine tuple is a $5$-tuple of the form
$$
(\ell,k,k^{\prime},mv,\ell^{\prime})
$$
where 
\begin{enumerate}
\item $\ell,\ell^{\prime}\in\omega$ are line numbers (also called states) together with the halt state denoted by $H$, 
\item $k, k^{\prime}$ are marks, i.e., elements of $\{0,1,2,3,4,1^{\star},2^{\star},3^{\star},4^{\star},\star,\star\star\}$,
\item $mv$ is a set Turing machine move. 
\end{enumerate}
A set Turing machine $M$ is a finite set of set Turing machine tuples.
\end{definition}

\noindent In the above definition, we think of the first two components of a set Turing machine tuple as an index to a command
consisting of the last three components of the tuple. 

% We have to define a configuration and then define how the Turing machine acts on a configuration. 

\begin{definition}
If $\beta$ is a limit ordinal and $\langle X_{\alpha}\mid \alpha<\beta\rangle$ is a transfinite sequence of 
set Turing machine tape markings, then by
$$X_{\beta}=\lim_{\alpha\rightarrow\beta} X_{\alpha}$$
we mean the pointwise limit, i.e., for each $\eta$ in $T$,
$$X_{\beta}(\eta)=\lim_{\alpha\rightarrow\beta} X_{\alpha}(\eta)$$
if this limit exists for every $\eta$. Otherwise the limit $X_{\beta}$ is not defined. 
\end{definition}

\begin{definition}
If $\beta$ is a limit ordinal and $\langle \eta_{\alpha}\mid \alpha < \beta\rangle$ is a sequence of elements from
$^{<\omega}ON$, then we say $\eta\in\, ^{<\omega}ON$ is weak upper bound for $\langle \eta_{\alpha}\mid \alpha < \beta\rangle$
if for every $\alpha<\beta$ there is a $\gamma$ such that $\alpha\leq \gamma<\beta$ and $\eta_{\gamma}\leq \eta$.
We define
$$
\underline{\lim}_{\,\alpha\rightarrow\beta}\,\eta_{\alpha}
$$
as the least $\eta$ such that $\eta$ is a weak upper bound for $\langle \eta_{\alpha}\mid \alpha < \beta\rangle$.
\end{definition}

\begin{definition}\label{configurations}
Let $X$ be a well formed marking and $M$ a set Turing machine. We define an ordinal $\alpha^*$ and a transfinite sequence
of markings 
$$\langle X_{\alpha}\mid \alpha\leq \alpha^*\rangle
$$ 
and configurations
$$
\langle (\alpha,\nu_{\alpha},X_{\alpha},\ell_{\alpha}) \mid \alpha\leq \alpha^*\rangle
$$
by induction on $\alpha$. 
We let 
$$
(0,\nu_0,X_0,\ell_0)=(0,0,X,0).
$$
If $\alpha=\gamma+1$, then if $X_{\gamma}(\nu_{\gamma})=k$ and the $(\ell_{\gamma},k)$-th entry of $M$
is 
$$
(k^{\prime},mv,\ell^{\prime})
$$
then $X_{\alpha}(\nu_{\gamma})=k^{\prime}$, $\ell_{\alpha}=\ell^{\prime}$, and
if $\nu_{\gamma}=(n_0,\alpha_1,\ldots,\alpha_s)$ then
\begin{enumerate}
\item If $mv=z$, then $\nu_{\gamma}=(n_0,\alpha_1,\ldots,\alpha_s,0)$.
\item If $mv=u$, then $\nu_{\gamma}=(n_0,\alpha_1,\ldots,\alpha_{s-1})$.
\item If $mv=+$, then $\nu_{\gamma}=(n_0,\alpha_1,\ldots,\alpha_s+1)$.
\item If $mv=u+$, then $\nu_{\gamma}=(n_0,\alpha_1,\ldots,\alpha_{s-1}+1)$.
\item If $mv=j+$, then $\nu_{\gamma}=(n_0+1,\alpha_1,\ldots,\alpha_s)$.
\item If $mv=j-$, then $\nu_{\gamma}=(n_0-1,\alpha_1,\ldots,\alpha_s)$  if defined, otherwise the machine halts.
\end{enumerate}
If $\alpha$ is a limit ordinal, then if 
$$\lim_{\gamma\rightarrow\alpha}X_{\gamma}$$ is defined then
the $\alpha$-th configuration is
$$
(\alpha,\nu_{\alpha},X_{\alpha},\ell_{\alpha})
$$
where
\begin{enumerate}
\item $\nu_{\alpha}=\underline{\lim}_{\,\gamma\rightarrow\alpha}\,\nu_{\gamma}$.
\item $X_{\alpha}=\lim_{\gamma\rightarrow\alpha}X_{\gamma}$.
\item $\ell_{\alpha}$ is the least state cofinal in $\langle \ell_{\gamma}\mid \gamma<\alpha\rangle$. 
\end{enumerate}
Otherwise the $\alpha$-th configuration is not defined.
We define $\alpha^*$ as the first $\alpha$ for which the above sequence is no longer defined or for which
$\ell_{\alpha}$ is the halt state. 
If the sequence of $X_{\alpha}$'s is unbounded in $ON$, we denote $\alpha^*=\infty$. 
We define $M(X)=Y$ if 
\begin{enumerate}
\item $X=X_0$.
\item $Y=X_{\alpha^*}$.
\item $Y$ is well formed.
\item $\ell_{\alpha*}=halt$.
\item $\nu_{\alpha^*}=0=0^{\frown}\langle\rangle$.
\end{enumerate}
In other words we require the machine to be in the halt state and the head position to be at the $0$-th root. 
$M(X)$ is undefined if $\alpha^*=\infty$. 
\end{definition}

\noindent We will be concerned with the operation of set Turing machines on well formed markings. The well formedness condition 
informs the set Turing machine as to when it is hitting a boundary of the representation of the set as marks on the tape.

\subsection{Examples of Set Turing Machines}

In this subsection, we give some illustrative examples of set Turing machines. Working through the examples will give the reader an intuition for how
set Turing machines operate. Our set Turing machines will be designed to work on well formed markings $X$. We define the end marking of
a marking $X$ as the first $\alpha$ such that $X(0^{\frown}\alpha)=0$. We define the machine $M_{end}$ which puts
the $\star\star$ mark on the end marking.
The machine
$M_{end}$ can be defined as he set Turing machine which  has the set of commands given in Table 1. Given a well formed marking 
$$
X=(X_0,\ldots,X_{n-1})
$$
which as all nonzero marks of $1$ and which has $n$-components, 
$M_{end}$ starts with the head position at the root of the first component $X_0$. 
It places a mark of $\star$ at the root. The $\star$ is used to mark the root as the starting  position so the machine will
know to halt when it returns to the root. It then proceeds through in $\prec_{lex}$ order all the nodes with cell addresses of the
form $0^{\frown}\alpha$ leaving cells marked with a $1$ unchanged. When it comes to the first $0$ in a cell it replaces the $0$ with the
$\star\star$ mark moves to the root and then halts. We use the $H$ symbol in the table for the $halt$ state.

% Requires the booktabs if the memoir class is not being used
\begin{table}[htbp]
   \centering
   %\topcaption{Table captions are better up top} % requires the topcapt package
   \begin{tabular}{ccc} % Column formatting, @{} suppresses leading/trailing space
      \toprule
      \multicolumn{2}{r}{Mark} \\
      \cmidrule(r){2-3} % Partial rule. (r) trims the line a little bit on the right; (l) & (lr) also possible
      State   & $0$ & $1$ \\
      \midrule
      $\ell_0$     &   & $(\star, z, \ell_1)$ \\
      $\ell_1$     & $(\star\star, u, H)$ & $(1, +, \ell_1)$ \\
      \bottomrule
   \end{tabular}
   \bigskip
   \caption{The set Turing machine $M_{{end}}$.}
   \label{tab:booktabs}
\end{table}

\FloatBarrier

%% Prec order is quite right. 

\noindent It is instructive to examine the behavior of $M_{end}$ at limit steps.  At step $n\in\omega$ the
machine head will be at cell $0^{\frown}n-1$. At time steps $\alpha$ such that $\omega\leq \alpha<\alpha^*$, it will be
at cell $0^{\frown}\alpha$, using the $\underline{\lim}$ operator at limit steps. 

$M_{end}$ has the function of delimiting the beginning and ends of the first component, and
it can be easily modified to do so for any particular component using the jump instruction. In what follows we will
usually operate on tapes with their nonzero marks delimited in this fashion, and may omit mentioning this assumption.

The next set Turing machine example is $M_{erase}$. Given a marking $X=(X_0,\ldots,X_{n-1})$ whose only 
nonzero markings are $1$'s we preprocess $X$ with $M_{end}$ to delimit the first component with the
start $\star$ and end marks $\star\star$. Let
$$
\Lambda_{X_0}=\left\{ \eta=(\eta_0,\ldots,\eta_{n-1})\mid X(0^{\frown}\eta)\neq 0\ \vee\ (X(0^{\frown}\eta)= 0\ \wedge\ X(0^{\frown}\eta\restriction (n-1)\neq 0)\right\}.
$$
In other words $\Lambda_{X_0}$ consists sequences from $^{<\omega}ON$ with associated marks all $1$'s or associated 
marks all $1$'s with a trailing $0$. 
$M_{erase}$ after the preprocessing of $X$ 
goes through the elements of the form $0^{\frown}\eta$ for $\eta\in \Lambda_{X_0}$ in the order of the
well ordering of $\Lambda_{X_0}$ induced by $\prec$. At time step $n<\omega$ the machine head will be
at the $n-1$-th element according to this order, and at time steps $\alpha\geq \omega$ the machine head will be at the
$\alpha$-th element as induced by $\prec$. The $\underline{\lim}$ operator on the head position in this instance reduces to
the $\lim$ operator as the sequence induced by $\prec$ on $\Lambda_{X_0}$ is strictly increasing. 
During this process all cells marked with a $1$ are now marked with a $0$.
The end mark $\star\star$ is also turned into a $0$. In other words, this set Turing machine erases the first component of $X$,
$X_0$, and turns all the $1$ marks to $0$. Once it
hits the $\star\star$ mark, it also turns it into a $0$, moves to the start and halts. Note that we can again using the jump operator modify $M_{erase}$ so that it erases
as many of the components of $X$ as we like. 

\begin{table}[htbp]
   \centering
   %\topcaption{Table captions are better up top} % requires the topcapt package
   \begin{tabular}{cccc} % Column formatting, @{} suppresses leading/trailing space
      \toprule
      \multicolumn{3}{r}{Mark} \\
      \cmidrule(r){2-4} % Partial rule. (r) trims the line a little bit on the right; (l) & (lr) also possible
      State   & $0$ & $1$ & $\star\star$ \\
      \midrule
      $\ell_0$     &   & $(\star, z, \ell_1)$ & \\
      $\ell_1$     & $(0, u+, \ell_1)$ & $(0, z, \ell_1)$ & (0,u,H)\\
      \bottomrule
   \end{tabular}
   \bigskip
   \caption{The set Turing machine $M_{erase}$.}
   \label{tab:booktabs}
\end{table}

\FloatBarrier

A related set Turing machine we give as an example is $M_{traverse-2}$. This machine acts much like 
$M_{erase}$ except as it traverses $\Lambda_{X_0}$ it replaces the $1$ marks by $2$s in the order induced by $\prec$ on
$\Lambda_{X_0}$.
 
\begin{table}[htbp]
   \centering
   %\topcaption{Table captions are better up top} % requires the topcapt package
   \begin{tabular}{cccc} % Column formatting, @{} suppresses leading/trailing space
      \toprule
      \multicolumn{3}{r}{Mark} \\
      \cmidrule(r){2-4} % Partial rule. (r) trims the line a little bit on the right; (l) & (lr) also possible
      State   & $0$ & $1$ & $\star\star$ \\
      \midrule
      $\ell_0$     &   & $(\star, z, \ell_1)$ & \\
      $\ell_1$     & $(0, u+, \ell_1)$ & $(2, z, \ell_1)$ & (0,u,H)\\
      \bottomrule
   \end{tabular}
   \bigskip
   \caption{The set Turing machine $M_{traverse-2}$.}
   \label{tab:booktabs}
\end{table}

\FloatBarrier

Our last example is of $M_{copy}$ which has the following table. 

\begin{table}[htbp]
   \centering
   %\topcaption{Table captions are better up top} % requires the topcapt package
   \begin{tabular}{cccc} % Column formatting, @{} suppresses leading/trailing space
      \toprule
      \multicolumn{3}{r}{Mark} \\
      \cmidrule(r){2-4} % Partial rule. (r) trims the line a little bit on the right; (l) & (lr) also possible
      State   & $0$ & $1$ & $\star\star$ \\
      \midrule
      $\ell_0$     &               & $(\star, z, 1)$ & \\
      $\ell_1$     & $(0, u+, \ell_1)$ & $(1, s, \ell_2)$          & (0, u, H)\\
      $\ell_2$     &               &   $(1, j^{+}, \ell_3)$  & \\
      $\ell_3$     &   $(1, j^{-}, \ell_4)$          &                       &\\
      $\ell_4$     &               &     $(1, z, \ell_1) $                 & \\
    \bottomrule
   \end{tabular}
   \bigskip
   \caption{The set Turing machine $M_{{copy}}$.}
   \label{tab:booktabs}
\end{table}

\FloatBarrier

The machine on input a preprocessed marking $X=(X_0,X_1)$ of $2$-components with $X_1$ representing the empty set and
$X_0$ which encodes a set, copies the marking onto the second
component. It does so by moving through $\Lambda_{X_0}$ in the $\prec_{lex}$ order and as it does so if the mark is a $1$
jumps to the corresponding node of the second component and placing a $1$ there, before jumping back and
continuing moving through $\Lambda_{X_0}$ with the $\prec_{lex}$ order. 
As in the case of $M_{erase}$, at limit steps, it arrives in state $1$ as that is the least cofinal state and continues on
as desired.  Note that the copy and erase machines are easily generalized to work on any desired component. 

\subsection{Constructions for Building Set Turing Machines}

\begin{definition}
A marking $X$ with a unique mark at $\eta$ is a marking $X$ such that 
$$X(\eta)\in \{0^{\star}, 1^{\star}, 2^{\star}, 3^{\star}, 4^{\star}\}$$
and all other marks in the component
of $\eta$ have marks from $0,1,2,3,4,\star,\star\star$. 
\end{definition}

\begin{definition}
We say the set Turing machine $M_B$ computes a Boolean if for every well formed marking $X$, $M_B(X)$ is defined and
the output of $M_B$ on input $X$ is the marking with one component whose only nonzero mark is at the root
and this mark is $1$, or if the only other mark is a $1$ in the $0^\frown 0$-th cell of the tape. In the first case we
say $M_B(X)=0$, in the second we say $M_B(X)=1$. 
\end{definition}

\begin{definition}
We say the set Turing machine $M$ preserves the number of components if for every well formed marking $X$ and $Y$ such
that $M(X)=Y$, 
$$
\#components(X)=\#components(Y).
$$
\end{definition}

\begin{definition}
Let $M$ and $M_B$ be  set Turing machines with $M_B$ computing a Boolean and $M$ preserving the number of components.
From any given marking $X$ and the sequence of markings $\langle X_{\alpha}\mid \alpha\leq \alpha^*\}$ determined from $X$ and $M$,
we construct the subsequence
determined by $M$ and $M_B$,
$$
\langle X_{\alpha}\mid \alpha\leq \beta^\},
$$
by letting $\beta^*$ be the first $\alpha\leq \alpha^*$ such that $M_B(X_{\alpha})=0$ or $\beta^*=\alpha^*$ if there is no such $\alpha$. 
\end{definition}

We will repeatedly use the following lemmas to justify our claims that certain set Turing machines with a specified desired functionality exist,
most especially the While Loop Lemma. 

\begin{lemma}(If Then Else Lemma)
If $M_1, M_2$ and $M_B$ are set Turing machines, with $M_1$ and $M_2$ preserving the number of components, and $M_B$ computing a Boolean, then there is a set Turing machine $N$ such that
for any well formed marking $X$,
\[
N(X) = 
\begin{cases}
M_1(X) & \text{if $M_B(X)=0$;}\\
M_2(X) & \text{if $M_B(X)\neq 0$.}\\
\end{cases}
\]
\end{lemma}

\begin{lemma}(Composition Lemma)
If $M_1$ and $M_2$ are set Turing machines, then there is a set Turing machine $N$ such that
for any well formed marking $X$,
$$
N(X)=M_2(M_1(X))
$$
for inputs $X$ such that $M_1(X)$ and $M_2(M_1(X))$ are defined. 
\end{lemma}

\begin{lemma}(While Loop Lemma)
Let $M$ and $M_B$ be set Turing machines with $M$ preserving the number of components and
$M_B$ computing a Boolean. Let $n\in\omega$. For each marking $X$ with $n$-components let
$\langle X_{\alpha}\mid \alpha\leq \alpha^*\rangle$ be the sequence determined from
$X$ by $M$ and $M_B$. 
Then there exists a set Turing machine $N$ such that for all markings $X$ with $n$ components, if 
$\alpha^*<\infty$, then
$$
N(X)=X_{\alpha^*}.
$$
and if $\alpha^*=\infty$, then $N(X)$ is undefined.
\end{lemma}
\begin{proof}
We form a loop iterating $M$ and exit the loop if $M_B(X_{\alpha})=0$. The core of the loop consists of the composition of 
$M$ with a machine that copies its output to components $n$ thru $2n-1$. We test whether to exit the loop by
applying $M_B$ to this copy and exit the loop if and only if the mark in cell $n\frown 0$ is $0$. 
\end{proof}

\begin{lemma}(Erase Below Unique Mark Lemma)
There is a set Turing machine $M$, such that if $X$ is a well formed tape with unique mark
$\eta$ in the first component, then $M(X)=Y$ with $X(\nu)=Y(\nu)$ for all $\nu$ not  below $\eta$, and with
$Y(\rho)=0$ for all $\eta\vartriangleleft\rho$, i.e., all marks below $\eta$ in $Y$ are $0$. 
\end{lemma}

\begin{lemma}(Subtree Copy Lemma)
Let $X$ be a well formed marking with $2$-components $X_0$ and $X_1$ and unique marks
$\eta_0$ in component $0$ and $\eta_1$ in component $1$. 
Then there is a set Turing machine $N$ whose output $Y$ has $2$-components, such that
$X_0=Y_0$ and $X_1(\nu)=Y_1(\nu)$ for all $\nu$ for which it is not the case that $\eta_1\vartriangleleft \nu$
and that ${Y_1}_{\eta_1}={X_0}_{\eta_0}$.  In other words, he machine copies that part of $X_0$ below $\eta_0$
to the part of $X_1$ below $\eta_1$ after first erasing the part of $X_1$ below $\eta_1$. 
\end{lemma}
\begin{proof}
We first let the machine operate as in the previous lemma to erase below $\eta_1$. We then place an end marker $\star\star$ to 
the supremum cell bounding the cells immediately below $\eta_0$. Once this set of preparations steps is finished, 
we then run a local version of a machine analogous to $M_{end+traverse}$ on 
the marks below $\eta_0$, but modify $M_{end+traverse}$ so that between moves we advance in the $\prec$ order 
both of the unique marks and copy whatever the unique mark
in the $0$-component is copied to the unique mark in the $1$-component. After copying a cell from the unique mark in
the $0$-component to the unique mark in the $1$-component, we advance the position of the unique mark in the
$1$-component in the style of $M_{end+traverse}$, before moving back to the unique mark in the $0$-component advancing it
and copying what mark we find over to the cell with unique mark in the $1$-component. 
\end{proof}

We will want to develop a characterization of set Turing computability in terms of the closure under certain computable operations
starting basic initial functions, much like the characterization of computable functions on $\mathbb{N}$. In the case of ordinary
computability it is easy to show that the relation $m<n$ among elements of $\mathbb{N}$ is computable by a Turing machine. 
The analogous fundamental relation for set Turing computability is the relation $x\in y$ among elements of $V$. The proof
that this relation is set Turing computable is not so straightforward. 

The point of the previous lemmas  and of the lemmas that follow is to provide tools that will enable use to prove that the $\in$-relation is computable by a set Turing machine. By making repeated use of the While Loop Lemma, we will prove a series of lemmas culminating in the Canonicalization Lemma, the main lemma we use to prove that the $\in$-relation is Turing computable. 

\begin{lemma}(Equality Lemma)
There is a set Turing machine $M_{=}$ which takes on Boolean values in the $2$-component
and decides if the first two components of the input marking are (literally) equal. 
\end{lemma}
\begin{proof}
First a $1$ mark is placed in the cell with address $2^{\frown}0$. The head of the set Turing machine traverses the first component according to the $\prec$ order. As it does this it jumps to the
corresponding cell in the second component with the $jump+$ operation, checks if the mark there
is the same as the mark in the cell it just came from, and then jumps back before continuing the traverse. If ever, the marks do
not agree, then a $0$ is placed in the cell with address $2^\frown 0$. 
\end{proof}

\begin{lemma}\label{existsequals}(Given an x there is a equal element of y lemma)
There is a set Turing machine $M_{\exists =}$ which given an input $X=(X_0,X_1)$ with two components takes on Boolean values in the component-$2$ and
decides if there is an ordinal $\alpha$ such that $X_0$ is equal to ${X_1}_{\alpha}$.
\end{lemma}
\begin{proof}
We use the while loop lemma applied to the machine $M_{=}$, but there are some slight modifications and other details to fill in. 
We use unique marks to mark to current element of $X_1$ to test, more formally which $\alpha$ to test. We then copy $X_1$ below $\alpha$ to
component-$3$. Within the loop, we apply $M_{=}$ to component-$0$ and component  $3$, letting the output of $M_{=}$
reside in component-$2$. We halt the loop if ever the cell at $2^{\frown}0$ gets a mark of $1$, then erase component-$3$.
\end{proof}

\begin{lemma}(For all x there is an equal y lemma)
There is a set Turing machine $M_{\forall x\exists y\,x=y}$ which takes on Boolean values in the component-$2$
and decides if for every element of the component-$0$ there is an element of the component-$1$ which is equal
(as markings) to the element of component-$0$. 
\end{lemma}
\begin{proof}
Similar to the previous proof except this time the while loop is formed around the machine $M_{\exists=}$. 
\end{proof}

\begin{lemma}(For all y there is an equal x lemma)
There is a set Turing machine $M_{\forall y\exists x \,x=y}$ which takes on Boolean values in the component-$2$
and decides if for every element of the component-$1$ there is an element of the component-$0$ which is equal
(as markings) to the element of component-$1$. 
\end{lemma}
\begin{proof}
Similar to the previous proof. 
\end{proof}

\begin{definition}
Let $X$ have a single component. Then $X$ is well marked with $2$'s and canonical below $\eta$, if $X_{\eta}$
has only nonzero marks of $2$ for cells below the root, and marked with $1$ at the root., and furthermore for all $\nu$ and $\rho$ such that
$\eta\vartriangleleft\nu$ and $\eta\vartriangleleft\rho$, if 
$$
Decode(X_{\rho})=Decode(X_{\nu})
$$
then in fact 
$$
X_{\nu}=X_{\rho}
$$
are literally identical as markings. 
\end{definition}

\begin{lemma}(Local Canonicalization Lemma)
There is a set Turing machine $M_{LC}$ such that 
for any well formed tape marking with single component $X$ which is canonical and marked with $2$'s below some rank $\alpha$,
outputs a marking $Y=M_{LC}(X)$, Y is canonical and marked with $2$'s below and at rank $\alpha$.
\end{lemma}
\begin{proof}
We begin by marking all nodes whose successors are all marked with $2$'s with a raised $^*$. Then we consider according to
the lexicographic ordering of  pairs induced by $\prec$ the set of such pairs. 
For each such pair we run $M_{\forall x\exists y\,x=y}$ and $M_{\forall y\exists x\,x=y}$ on the copies of the pair
(in component-$1$ and component-$2$). If both return a $1$, then we replace marking below the last node of the pair (according
to the $\prec$ order) with markings of the first node of the pair. Once all such pairs of nodes have been processed the output
has the properties we seek, canonicalization for all nodes of rank $\alpha$.  To finish, we replace the marks with a raised star,
$^*$ with $2$'s. 
\end{proof}

\begin{lemma}(Canonicalization Lemma)
There is a set Turing machine $M_C$ such that for any well formed tape marking with single component $X$, with no
marks of $2$, outputs a marking $\tilde{X}=M_C(X)$ such that if $\nu$ and $\eta$ are two tape positions such that
$$
Decode(\tilde{X}_{\nu})=Decode(\tilde{X}_{\eta})
$$
then in fact 
$$
\tilde{X}_{\nu}=\tilde{X}_{\eta}
.$$
\end{lemma}
\begin{proof}
We begin an initialization procedure by labeling all nodes $\eta$ 
such that all successor node $\eta\vartriangleleft\nu$ of $\eta$ having a mark of $0$, 
$X(\eta)=0$, with a mark of $2$. Let $X_0$ be this new marking. Notice that $X_0$ is canonical and marked with $2$'s 
at rank $0$. By the While Loop Lemma, by forming a loop around the set Turing machine $M_{LC}$, we can produce a transfinite sequence
$\langle X_{\alpha}\mid \alpha\leq \alpha^*\rangle$ such that $X_{\alpha}$ is canonical and well marked with $2$'s below and at $\alpha$ where $\alpha^*$ is the rank of $X$. Denoting $M_C$ as the composition of the initialization with the While Loop Lemma, then
$$M_C(X)=X_{\alpha^*}=\tilde{X}.$$
\end{proof}

\noindent For a marking $X$, we use the notation $\tilde{X}$ for its canonicalization.

\begin{lemma}(Pair Lemma)
Let $X=(X_0,X_1)$ be a well formed marking with two components. Then there are set Turing machines, $M_{pair}$ and
$M_{opair}$,
which on input $X$ outputs $Y=(Y_0,Y_1,Y_2)$ and $Z=(Z_0,Z_1,Z_2)$ respectively, such that
\begin{enumerate}
\item $Z_0=Y_0=X_0$
\item $Z_1=Y_1=X_1$
\item $Decode(Y_2)=\{Decode(Y_0,Y_1)\}$
\item $Decode(Z_2)=(Decode(Z_0),Decode(Z_1))$.
\end{enumerate}
\end{lemma}

\begin{lemma}(Pairing Lemma)
There is a set Turing machine $M_{pairing}$, such given a well formed input $X=(X_0,X_1)$ with two components,
such that if $\alpha_0$ is the least ordinal such that such that $X_0(\alpha_0)=0$, and $\alpha_1$ is the least
ordinal such that $X_1(\alpha_1)=0$, and if $\alpha_0=\alpha_1$, then the machine outputs a marking $Y$ with
three components $(Y_0,Y_1,Y_2)$ such that 
\begin{enumerate}
\item $X_0=Y_0$
\item $X_1=Y_1$
\item $Decode(Y_2)$ is a one to one mapping of $Decode(X_0)$ to $Decode(X_1)$. 
\end{enumerate}
\end{lemma}

\noindent Note that the machine $M_C$ does not make representations of arbitrary sets via marks globally canonical. However,
for ordinals, there is a notion of canonical representation. These will be important in our proof that REC $\subseteq$ TUR.

\begin{definition}
Let $\alpha$ be an ordinal. By induction on $\alpha$, we define the canonical representative (canonical name) for $\alpha$, denoted
$\alpha_C$. 
\begin{enumerate}
\item $0_C$ is the marking whose only nonzero marking is at cell $0$, i.e., 
$$0_C(0)=1.$$
\item $1_C$ is the marking whose only nonzero markings are at cells $0$ and $0^\frown 0$, 
$$1_C(0)=1,\ \ 1_C(0^\frown0)=1.$$
\item In general, we let $\alpha_C$ be the name such that
$$
\alpha_C(0\frown \beta)=\beta_C
$$
for $\beta<\alpha$ and
$$
\alpha_C(0\frown \alpha)=0.
$$
\end{enumerate}
\end{definition}

\begin{lemma}(Canonical Successor Lemma)
There is a set Turing machine, $M_{succ}$ such that for any ordinal $\alpha$, if $M_{succ}$ is given  input
$\alpha_C$, then the  output is ${\alpha+1}_C$.
\end{lemma}
\begin{proof}
$M_{succ}$ first copies the input $X$ to component-$1$. Let $\alpha$ be the first
cell of the form $0^\frown\alpha$ such that $X(0^\frown \alpha)=0$. Then it places a $1$ in cell $0^\frown\alpha$
and copies component-$1$ onto the space below $0^\frown\alpha$. Lastly, we erase component-$1$, so the output
is ${\alpha+1}_C$ if the input $X$ is $\alpha_C$ for some $\alpha$. 
\end{proof}

\begin{lemma}(Canonical Ordinal Lemma)
There is a set Turing machine $M_{\alpha_C}$, that on input an representation $X$ of a set $x$,
outputs the canonical representation of an ordinal $\alpha_C$ for some ordinal $\alpha$ which has the
same cardinality as $x$. 
\end{lemma}
\begin{proof}
We begin by letting $\alpha$ be the strict sup of all $\beta$ such that $X(0^\frown\beta)=1$. We place
an end mark of $**$ on cell $0^{\frown}\alpha$. 
Let $M_{succ}$ be the Turing machine from the previous lemma. 
Using the While Loop Lemma around the machine $M_{succ}$ we can compute the sequence of markings
$$
\langle \beta_C\mid \beta<\alpha^*\rangle
$$
where $\alpha^*=\alpha$, halting the loop when machine head reaches the $**$ node. 
\end{proof}

\begin{lemma}(Canonical Well Ordering by an Ordinal Lemma)
There is a set Turing machine $M_{cwo}$, that on input a representation $X$ of a set $x$, outputs
a marking $X_{cwo}$ which represents a well ordering of $x$ by an ordinal.
\end{lemma}
\begin{proof}
We begin by using $M_{\alpha_C}$ to produce a marking $Y=(Y_0,Y_1)$ with two components, $Y_0=X$
and $Y_1=\alpha_C$ for some ordinal $\alpha$ with the same cardinality as $x$. We then apply the machine $M_{pairing}$
from the pairing lemma to $Y$ to get a third component, $Y_2$ such that $Decode(Y_2)$ is a well ordering of $x$ by the ordinal
$\alpha$. 
\end{proof}

\subsection{The Turing Computable Class Functions {\rm TUR}}

\begin{definition}
Let $x$ be a set. A well ordering of $x$ by an ordinal $\alpha$ is a bijection $f$ between $x$ and $\alpha$.
Similarly, if $\bar{x}$ is an $n$-tuple of sets, then a well ordering of $\bar{x}$ by ordinals $\bar{f}$ is a finite sequence
of well orderings $(f_0,\ldots,f_{n-1})$, with each $f_i$ a well ordering of $x_i$ by some ordinal $\alpha_i$. 
We say $\bar{f}$ is a well ordering of the transitive closure of $\bar{x}$ if for each $i<n$,
$f_i$ is a well ordering of the transitive closure of $x_i$. 
\end{definition}

\begin{definition}
If $\bar{x}$ is a $n$-tuple of sets and $\bar{f}$ is a well ordering of transitive closure of $\bar{x}$, then we define the function
$CCodeBy(\bar{f},\bar{x})=\bar{z}$ we mean $\bar{z}$ is the canonical well formed  marking representing $\bar{x}$ as
induced by $\bar{f}$.  We define the relation $CCode(\bar{x},\bar{z})$, read $\bar{z}$ is a code for $\bar{x}$,
if and only if for some $\bar{f}$, $CCodeBy(\bar{f},\bar{x})=\bar{z}$.
\end{definition}

\begin{definition}
If $M$ is a set Turing machine, then we define the partial class function $F_M$ with ${\rm dom}\,F_M\subseteq V^n$ as follows.
If $\bar{x}$ is an $n$-tuple of sets, then we define
$$
F_M(\bar{x})=y
$$
if and only if for every $\bar{z}$ such that $CCode(\bar{x},\bar{z})$ and $w$ such that $M(\bar{z})=w$, 
$$DeCode(w)=y.$$
\end{definition}

\begin{definition}
{\rm TUR} is the collection of all partial class functions of the form $F_M$ where $M$ is a set Turing machine. Elements
of {\rm TUR} are said to be Turing computable. A relation is said to be Turing computable it is characteristic function is
Turing computable. 
\end{definition}

\begin{theorem}
The relations $x\in y$ and $x=y$ are  Turing computable. 
\end{theorem}
\begin{proof}
This follows almost immediately from the Canonicalization Lemma. Given two sets $x$ and $y$ and corresponding encodings
$X$ and $Y$, form a coding $Z$ for the pair $\{x,y\}$ by letting $Z_{0^\frown 0}=X$ and $Z_{0^\frown 1}=Y$. Let $Z^{\prime}$
be the canonicalization of $Z$ with components $X^{\prime}$ and $Y^{\prime}$. Then $x\in y$ if and only if there is an ordinal $\alpha$
such that $Y^{\prime}_{\alpha}=X^{\prime}$. Similarly for the $x=y$ relation. 
\end{proof}

\section{The Recursive Class Functions REC}

\begin{definition}
The initial class functions consist of the following functions.
\begin{enumerate}
\item $F(x)=0$.
\item  $P_{n,i}(x_1,\ldots,x_n)=x_i$ for $1\leq i\leq n<\omega$.
\item $F(x,y)=x\cup\{y\}$.
\item The class function $C(x,y,u,v)$ where
\[
C(x,y,u,v)=
\begin{cases}
u & \text{ if $x\in y$},\\
v & \text{otherwise.}\\
\end{cases}
\]
\end{enumerate}
\end{definition}

\begin{definition}
Let $x$ be a set. We denote the set of well orderings of $x$ by an ordinal
as $woo(x)$. 
\end{definition}

\begin{definition}
The basic class operations consist of the following.
\begin{enumerate}
\item Composition:- If $G_1$, $G_2$, and $H$ are class functions of $n+1$, $n+m+1$, and $m$ variables respectively,
then the composition of $G_1$ with $H$ and the composition of $G_2$ with $H$ yields the class functions
$$
F_1(x_1,\ldots,x_m,y_1,\ldots,y_n)=G_1(H(x_1,\ldots,x_m),_1,\ldots,y_n)
$$
and
$$
F_2(x_1,\ldots,x_m,y_1,\ldots,y_n)=G_2(x_1,\ldots,x_m,H(x_1,\ldots,x_m),y_1,\ldots,y_n)
.$$
\item Recursion: - If $G$ is a class function of $n+2$ variables, then recursion on $G$ yields the class function
$$
F(x_1,\ldots,x_n,z)=G(\cup\{F(x_1,\ldots,x_n,u)\mid u\in z\},x_1,\ldots,x_n,z).
$$
\item $\mu$-operator: - If $G$ is a class function of $n+1$ variables then the $\mu$-operator applied to
$G$ yields the function $F=\mu G$ where
$$
F(x_1,\ldots,x_n)=\alpha
$$
if and only if $G(x_1,\ldots,x_n,\alpha)=0$, and for all $\beta<\alpha$, $G(x_1,\ldots,x_n,\beta)$ is defined
and $G(x_1,\ldots,x_n,\beta)\neq 0$. Otherwise $F(x_1,\ldots,x_n)=\mu G(x_1,\ldots,x_n)$ is not defined. 
\item Random Well Orderings by Ordinals:- If $G$ is a class function on $2n$ variables such that for all
$x_1,\ldots,x_n,y\in V$, and $\bar{f}=(f_1,\ldots,f_n)$,
$$
\forall \bar{f}(\bigwedge_{i\leq n}f_i\in woo(x_i)\ \rightarrow\  G(x_1,\ldots,x_n,f_1,\ldots,f_n)=y)
$$
if and only if
$$
\exists \bar{f}(\bigwedge_{i\leq n}f_i\in woo(x_i)\ \wedge\  G(x_1,\ldots,x_n,f_1,\ldots,f_n)=y)
$$
then random well orderings by ordinals applied to $G$ yields the class function $F(x_1,\ldots,x_n)$
where
$$
F(x_1,\ldots,x_n)=y
$$
if and only if
$$
\exists \bar{f}(\bigwedge_{i\leq n}f_i\in woo(x_i)\ \wedge\  G(x_1,\ldots,x_n,f_1,\ldots,f_n)=y)
.$$
\end{enumerate}
\end{definition}

\begin{definition}
The primitive recursive set functions, denoted pREC, are the smallest collection of class functions, containing
all the initial class functions and closed under the operations of Composition and Recursion. The min recursive set
functions, denoted minREC, are the smallest collection of class functions, containing the initial class functions and
closed under the operations of Composition, Recursion, and the $\mu$-operator. Finally, REC, is the class of recursive
set functions, which are the smallest collection of class functions, containing the initial class functions and closed under
the operations of Composition, Recursion, the $\mu$-operator, and Random Well Orderings by Ordinals. 
\end{definition}

The primitive recursive set functions were introduced by Jensen and Karp in \cite{JK}. 

\begin{definition}
Let $R(x_1,\ldots,x_n)$ be a relation on $V^n$. We say $R(x_1,\ldots,x_n)$ is set recursive if its 
characteristic function is. 
\end{definition}

\begin{definition}
A formula in the language of set theory is called $\Delta_0$ if all its quantifiers are restricted, i.e., occur in the form
$\forall x(x\in y\ \rightarrow \ldots)$ or $\exists x(x\in y\ \rightarrow\ldots)$. A formula $\phi(x)$ is $\Sigma_1$ if it has the form
$$
\exists y\psi(x,y)
$$
where $\psi(x,y)$ is a $\Delta_0$ formula. A formula $\phi(x)$ is $\Pi_1$ if it has the form
$$
\forall y\psi(x,y)
$$
where $\psi(x,y)$ is a $\Delta_0$ formula. A formula is $\Delta_1$ if it is equivalent in $V$ to be a $\Sigma_1$ and
a $\Pi_1$ formula. 
\end{definition}

\begin{remark}
All $\Delta_0$ functions are class computable and all computable class functions are $\Delta_1$. In fact, if
$F$ is class computable, then there is a canonical $\Sigma_1$ formula $\sigma_F$ and a canonical $\Pi_1$
formula $\pi_{F}$ both which define $F$. If $V=L$, since the canonical well ordering of $L$ is class computable,
then the computable class functions and the $\Delta_1$ functions coincide.
\end{remark}

Let $H(\kappa)$ denote the sets of hereditary cardinality less than $\kappa$.  Then set recursive functions behave
properly on sets in $H(\kappa)$. 

\begin{lemma}
Let $\kappa$ be an uncountable cardinal and $F$ a computable class function. Then for each $x_0,\ldots,x_{n-1}\in H(\kappa)$,
$$
F(x_0,\ldots,x_{n-1})\in H(\kappa).
$$
\end{lemma}
\begin{proof}
Let $x_0,\ldots,x_{n-1}\in H(\kappa)$ and let $\sigma_F$ be the canonical $\Sigma_1$ formula defining $F$. 
Let $F(x_0,\ldots,x_{n-1})=x_n$. Let $N$ be a set of cardinality less than $\kappa$ such that for each $i\in n$, 
$x_i\in N$, and $trcl\, x_i\subseteq N$, and $N$ reflects $\sigma_F(v_0,\ldots,v_n)$. Let $M$ be the transitive
collapse of $N$. By reflection, 
$$
N\models \sigma_F(x_0,\ldots,x_{n-1},x_n).
$$
Since $M$ is isomorphic to $N$ by a map taking $x_i$ to $x_i$ for $i\in n$, there is some $u\in M$, such that
$$
M\models \sigma_F(x_0,\ldots,x_n,u)
$$
with $u\in H(\kappa)$ since $M\subseteq H(\kappa)$. By the upward absoluteness of $\Sigma_1$ formulas,
$$
V\models \sigma_F(x_0,\ldots,x_{n-1},u)
$$
and since $\sigma_F(v_0,\ldots,v_{n-1},v_n)$ defines $F$, $u$ must be equal to $x_n$. 
\end{proof}

\section{TUR = REC}

\subsection{TUR$\subseteq $ REC}

\begin{lemma}(Basic Codes are Set Recursive)
The function $BasicCode(w,v)$ which on input a set $v$ and a set $w\in woo(trcl(\{v\})$, outputs
a basic code for $v$, is set recursive. 
\end{lemma}
\begin{proof}
First note that the function that takes a set to its transitive closure is set recursive as
$$
trcl\,x=x\cup\bigcup\{trcl\, y\mid y\in x\}.
$$
We next define an auxiliary set recursive function $f$ which will be used in the proof. An element of the range of $f$ will be
a subset of  $trcl(\{v\})\times ^{<\omega}ON$ representing a function. By induction on $\alpha$ we
define $f$ by letting for a set $v$, and a well ordering by an ordinal $w$ of $trcl\, \{v\}$, $f(\alpha,w,trcl\,\{v\})$
be defined by
\begin{enumerate}
\item $f(0,w,trcl\,\{v\})=\{(v,\langle\rangle)\}$
\item For $\alpha$ a limit ordinal,
$$
f(\alpha,w,trcl\,\{v\})=\bigcup_{\beta<\alpha}\,f(\beta,w,trcl\,\{v\})
$$
\item Let 
$$
f(\alpha+1,w,trcl\,\{v\})=f(\alpha,w,trcl\,\{v\})\cup \{(u,\rho)\}
$$
where $\rho$ is a finite sequence of ordinals and $u$ is the least element of $trcl\,\{v\}$ in the sense of $w$ such that
\begin{enumerate}
\item $u\not\in {\rm{dom}}\, f(\alpha,w,trcl\,\{v\})$
\item There is a finite sequence $v_0,\ldots,v_n$ such that
\begin{enumerate}
\item $v_0=v$
\item $v_n=u$
\item $v_n\in v_{n-1}\in\ldots\in v_0$
\item For all $i<n$, $v_i\in {\rm{dom}}\,f(\alpha,w,trcl\,\{v\})$.
\end{enumerate}
\item $length(\rho)=n$
\item $\rho\restriction n\in {\rm{ran}}\,f(\alpha,w,trcl\,\{v\})$
\item $\rho(n)$ is the least ordinal $\gamma$ such that $\rho\restriction n^\frown\gamma\not\in {\rm ran }f(\alpha,w,trcl\,\{v\})$.
\end{enumerate}
\end{enumerate}

\noindent If $\zeta$ is the least ordinal such that $f(\zeta,w,trcl\,\{v\})=f(\zeta+1,w,trcl\,\{v\})$, i.e.
$$
\zeta=\mu\alpha\left(f(\alpha,w,trcl\,\{v\})=f(\alpha+1,w,trcl\,\{v\})\right),
$$
then
$$
BasicCode(w,trcl\,\{v\})={\rm{ran}}\,f(\zeta,w,trcl\,\{v\}).
$$
\end{proof}

Following from the set recursiveness of the BasicCode function, it follows that the code function $Code(x_1,\ldots,x_n,w_1,\ldots,w_n)$ 
where $x_i\in woo(w_i)$, is set recursive since
$$
Code(x_1,\ldots,x_n,w_1,\ldots,w_n)=(BasicCode(x_1,w_1),\ldots,BasicCode(x_n,w_n))
.$$

\begin{lemma}(Decode Function Lemma)
The Decode function is set recursive. 
\end{lemma}
\begin{proof}
We first define an auxiliary set recursive function $g$ which has in its domain pairs of the form $(\alpha,tr)$ with $\alpha$ an ordinal
and $tr$ be a well founded subset of $^{<\omega}ON$, and has as its  range, functions from finite sequences of ordinals to sets.
For $tr$ a well founded subset of $^{<\omega}ON$, 
\begin{enumerate}
\item Let $g(0,tr)=\{(\nu,\emptyset)\mid \nu\in tr\ \wedge\ \neg\exists \rho(\rho\in tr\ \wedge\ \nu\vartriangleleft\rho)$
\item If $\alpha$ is a limit ordinal, then
$$
g(\alpha,tr)=\bigcup_{\beta<\alpha}g(\beta,tr).
$$
\item $g(\alpha+1,tr)=g(\alpha,tr)\cup$
$$
\bigcup\big\{(\nu,s_{\nu})\mid \nu\not\in {\rm{dom}}\, g(\alpha,tr)\ \wedge\ \forall\rho\big((\rho\in tr\ \wedge\ \rho\restriction length(\eta)=\nu)\Rightarrow
\rho\in {\rm{dom}}\,g(\alpha,tr)\big)\big\}
$$
where 
$$
s_\nu=\bigcup\big\{\{g(\alpha,tr)(\rho)\}\mid \rho\in tr\ \wedge\ \nu\vartriangleleft\rho\ \wedge\ \ length(\rho)=length(\nu)+1\big\}
.$$
\end{enumerate}
Let
$$
\gamma=\mu\alpha(g(\alpha,tr)=g(\alpha+1,tr)).
$$
Then
$$Decode(tr)=g(\gamma,tr)(\langle\rangle).$$
\end{proof}

\begin{theorem}
TUR$\  \subseteq \ $REC.
\end{theorem}
\begin{proof}
Let $F(v_0,\ldots,v_{n-1})\in \rm{TUR}$ and let $M$ be a set Turing machine witnessing this fact. Let $x_0,\ldots,x_{n-1}\in V^n$
and let $w_i\in woo(x_i)$. Let $X_i=BasicCode(w_i,x_i)$ and let $X=(X_0,\ldots,X_{n-1})$. By transfinite induction on $\alpha$
we define the $\alpha$-th configuration of the machine $M$ on input $X$. The $\alpha$-th configuration consists of a
marking of the tape $X(\alpha)$, a position of the Turing machine head $p(\alpha)$, and a state $s(\alpha)$. 
We define for $\alpha\leq \alpha^*$ where $\alpha^*$ is the least ordinal such that $s(\alpha)=halt$ or 
configurations past the $\alpha^*$-th are not defined as in definition \ref{configurations}.
Without loss of generality $X(\alpha^*)$ has a single component and we can define 
$$F(x_0,\ldots,x_{n-1})=Decode(X(\alpha^*)).$$
By assumption on $F$, $F(x_0,\ldots,x_{n-1})$ and is independent of code $X_i$ for $x_i$ and so does not depend on the choice of 
the well ordering by ordinals $w_i$ of $x_i$. Since the encoding and decoding functions are set recursive, as is the transfinite
sequence of configurations, and as the answer is independent of the choice of the $w_i$, by closure under the random well ordering by ordinals clause 
we have that $F$ is set recursive. 
\end{proof}

\subsection{REC$\,\subseteq\ $TUR}

To prove that REC$\ \subseteq\ $TUR it is enough that the initial functions are in TUR, and that TUR is closed under 
composition, definition by recursion, closed under the $\mu$-operator, and random well orderings by ordinals.

\begin{lemma}
TUR contains all the initial set recursive functions.
\end{lemma}
\begin{proof}
That the first three initial functions are TUR is covered by the set Turing machines defined in Section 2 or variations of them.
The proof that $C(x,y,u,v)$ is in TUR follows from the results in Section 2, in particular the canonicalization lemma. It is enough to prove that there is a 
Turing machine $M_{x\in y}$ that outputs a Boolean for deciding if $x\in y$, since from $M_{x\in y}$ we can copy either
the representative for $u$ or $v$ to a separate component, erase all the other components, and then copy the representative of
$u$ or $v$ to the first component,(component-$0$). Let $X$ be a representative for $x$ and $Y$ a representative for $y$. 
Form in component-$2$, a representative $Z$ for $\{X,Y\}$, and then form its canonicalization $\tilde{Z}$. Now $x\in y$ if and only
if for some ordinal $\alpha$, $Decode(\tilde{Z}_{0})=Decode(\tilde{Z}_{1\frown\alpha})$, which is decidable by a set Turing machine
of the form $M_{\exists =}$ from lemma \ref{existsequals}.
\end{proof}

\begin{lemma}
TUR is closed under composition.
\end{lemma}
\begin{proof}
Same as the proof of closure under composition for ordinary Turing machines.
\end{proof}

\noindent To prove that TUR is closed under recursion we will need some definitions and lemmas.

\begin{definition}
Let $X$ be a well formed Turing name and $\eta$ a node in $X$ such that $X(\eta)\neq 0$. Then the storage node associated with 
$\eta$, denoted $s_{\eta,X}$ is the node $\eta^\frown\alpha^{**}$ where $\alpha^{**}$ is the least $\alpha$ such that
$X(\eta^\frown\alpha)=0$.
\end{definition}

\noindent We need to show that TUR is closed under recursion by $F$ and $G$ given both $F$ and $G$ are in TUR.
We will have to compute 
$$
F(x_1,\ldots,x_n,z)=G(\cup\{F(x_1,\ldots,x_n,u)\mid u\in z\},x_1,\ldots,x_n,z)
$$
given $\{F(x_1,\ldots,x_n,u)\mid u\in z\}$. 
 First another definition.

\begin{definition}
Let $F(v_1,\ldots,v_n,v)$ be a set function and $X_1,\ldots,X_n$ names and $Z$ a name. For a name $Z$, let  $\alpha^*_Z$ be the least ordinal
such that $Z(\alpha^*)\neq 0$. We say a name $Z^*$ is 
a name for $Z$ adorned by $F$ if for each $\alpha<\alpha^*$, $Z^*_{s_{\alpha,Z}}$ is a name
for $F(v_1,\ldots,v_n,Z_{\alpha})$. 
\end{definition}

\begin{lemma}(Local Recursion Step Lemma)
Let $x_1,\ldots,x_n,z$ be sets with Turing names $X_1,\ldots,X_n,Z$. Let $G$ be in TUR and let $Z^*$ be a name
name for $Z$ adorned by $F$.
Then there is a set Turing machine $M_{s\ell,G}$ which on input $(X_1,\ldots,X_n,Z^*)$ outputs
a name $Z^{**}$ such that $Z^{**}(\alpha^{*}_{Z^{*}})$ is a name for 
$$F(x_1,\ldots,x_n,z)=G(\cup\{F(x_1,\ldots,x_n,u)\mid u\in z\},x_1,\ldots,x_n,z).$$
\end{lemma}
\begin{proof}
Let $M_G$ be a machine for $G$. The machine $M_{s\ell,G}$ copies the adornments by $F$ onto the $n+2$-th component  to create a name for  $$\cup\{F(x_1,\ldots,x_n,u)\mid u\in z\}$$
at the $n+2$-th component.  It then creates a name for
$$
G(\cup\{F(x_1,\ldots,x_n,u)\mid u\in z\},x_1,\ldots,x_n,z)
$$
by applying $M_G$ at the $n+2$-th component before copying this  name onto the storage node
$$s_{\langle\rangle,Z^*}=\alpha^{*}_{Z^{*}}.$$ 
To finish it then erases the $n+2$-th component. 
\end{proof}

\begin{lemma}(Local Recursion Lemma)
Let $F$ and $G$ be in TUR. There is a set Turing machine $M_{\ell,F,G}$ such that if $X_1,\ldots,X_n,Z$ are names
and $\alpha$ is an ordinal such that all nodes $\eta\in Z$ with rank less than $\alpha$ are
adorned with a name for 
$$F(Decode(X_1),\ldots,Decode(X_n),Decode(Z_{\eta})),$$
outputs a name $Z^*$ with the property that for all nodes $\eta\in Z$ with rank less than or equal to $\alpha$ are
adorned with a name for $F(Decode(X_1),\ldots,Decode(X_n),Decode(Z_{\eta}))$.
\end{lemma}
\begin{proof}
We use While Loop Lemma around a loop using the  machine from the Local Recursion Step Lemma $M_{s\ell,G}$,
to all the nodes of Rank $\alpha$ in $Z$.
\end{proof}

\begin{lemma}
TUR is closed under recursion.
\end{lemma}
\begin{proof}
Let $X_1,\ldots,X_n,Z$ be names. 
By another use of the While Loop Lemma, this time built around a loop using $M_{\ell,F,G}$ from the Local Recursion Lemma,
we can build a name $Z^*$ which is completely adorned for all nodes $\eta$ of $Z$. The value of the adornment of the
root node of $Z^*$ is a name for $F(Decode(X_1),\ldots,Decode(X_n),Decode(Z_{\eta}))$.
\end{proof}

\begin{lemma}
TUR is closed under the $\mu$-operator.
\end{lemma}
\begin{proof}
Let $G(x_1,\ldots,x_{n},x_{n+1})$ be in TUR, computable by the set Turing machine $M_G$. We define a machine $M$ computing
$$
\mu\alpha(G(x_1,\ldots,x_n,\alpha).
$$
We can assume $M_G$ puts its output in
the $n+2$-th component. Let $X_1,\ldots,X_{n}$ be names for $x_1,\ldots,x_n$. We use the While Loop Lemma.
At iteration $\alpha$ of the loop starting from iteration
$\alpha=0$ we build a canonical representative for $\alpha$, $\alpha_C$. The machine halts the construction if at any point $M_G(X_1,\ldots,X_{n},\alpha_C)=0_C$, and then
output $\alpha_C$ in this case in the first component, erasing the other components. 
\end{proof}

\begin{lemma}
TUR is closed under random well ordering by ordinals.
\end{lemma}
\begin{proof}
The main point here is that for a code $X$ for a set $x$, one can modify the set Turing machine that builds a canonical name 
$\alpha_C$ for an
ordinal $\alpha$ with $|\alpha|=|x|$, and simultaneously build a name for a well ordering between $\alpha$ and $x$. Namely,
we can let $\alpha$ be the least ordinal such that $X(0\frown\alpha)=0$. 
\end{proof}

\begin{corollary}
REC$\ \subseteq\ $TUR.
\end{corollary}

\end{document}